\documentclass[11pt,oneside,reqno]{article}
\usepackage[margin=1.3in]{geometry}

\usepackage{amsmath,amsthm,amssymb,color}
\usepackage[authoryear]{natbib}
\usepackage{booktabs}

\RequirePackage{amsthm,amsmath,amsfonts,amssymb,mathrsfs,dsfont,mathtools,thmtools}
\RequirePackage{natbib}
\RequirePackage[colorlinks,linkcolor=blue,citecolor=blue,urlcolor=blue]{hyperref}
\usepackage[capitalize]{cleveref}
\RequirePackage{graphicx}

\crefname{equation}{}{}
\crefformat{equation}{\textup{(#2#1#3)}}
\crefrangeformat{equation}{\textup{(#3#1#4)--(#5#2#6)}}
\crefdefaultlabelformat{#2\textup{#1}#3}
\crefname{lemma}{Lemma}{Lemmas}
\crefname{page}{p.}{pp.}

\usepackage{bbm}
\usepackage{mathrsfs}
\usepackage{graphicx}
\usepackage{tikz}
\usepackage{enumitem}
\usetikzlibrary{arrows, automata}
\usetikzlibrary{calc}
\usetikzlibrary{positioning}

\numberwithin{equation}{section}
\allowdisplaybreaks[4]

\theoremstyle{plain}
\newtheorem{theorem}{Theorem}[section]

\newtheorem{lemma}{Lemma}[section]

\theoremstyle{definition}
\newtheorem{definition}{Definition}[section]

\newtheorem{remark}{Remark}[section]

\makeatletter

\newcount\minute
\newcount\hour
\newcount\hourMins
\def\now{%
\minute=\time%
\hour=\time \divide \hour by 60%
\hourMins=\hour \multiply\hourMins by 60%
\advance\minute by -\hourMins%
\zeroPadTwo{\the\hour}:\zeroPadTwo{\the\minute}%
}
\def\zeroPadTwo#1{\ifnum #1<10 0\fi#1}

\renewcommand{\cite}{\citet}

\def\^#1{\ifmmode {\mathaccent"705E #1} \else {\accent94 #1} \fi}
\def\~#1{\ifmmode {\mathaccent"707E #1} \else {\accent"7E #1} \fi}

\def\*#1{#1^\ast}
\edef\-#1{\noexpand\ifmmode {\noexpand\bar{#1}} \noexpand\else \-#1\noexpand\fi}
\def\>#1{\vec{#1}}
\def\.#1{\dot{#1}}

\def\atop{\@@atop}
\def\*#1{\mathscr{#1}}

\renewcommand{\leq}{\leqslant}
\renewcommand{\geq}{\geqslant}

\newcommand{\eq}{\eqref}

\newcommand{\IE}{\mathbbm{E}}
\newcommand{\IP}{\mathbbm{P}}
\newcommand{\Var}{\mathop{\mathrm{Var}}\nolimits}

\newcommand{\IZ}{\mathbbm{Z}}
\newcommand{\IR}{\mathbb{R}}

\def\be#1{\begin{equation*}#1\end{equation*}}
\def\ben#1{\begin{equation}#1\end{equation}}
\def\bes#1{\begin{equation*}\begin{split}#1\end{split}\end{equation*}}
\def\besn#1{\begin{equation}\begin{split}#1\end{split}\end{equation}}

\def\beqn#1\eeqn{\begin{align}#1\end{align}}
\def\beq#1\eeq{\begin{align*}#1\end{align*}}

\usepackage{graphicx}
\usepackage{latexsym}
\usepackage{amsmath,amsthm,amssymb,amscd}
\usepackage{epsf,amsmath}

\def\E{{\IE}}
\def\P{{\IP}}

\def\blfootnote{\xdef\@thefnmark{}\@footnotetext}

\makeatother

\begin{document}

\title{
Edgeworth Expansion by Stein's Method
}
\author{Xiao Fang and Song-Hao Liu}
\date{\it \small The Chinese University of Hong Kong and Southern University of Science and Technology} 
\maketitle

\noindent{\bf Abstract:} 
Edgeworth expansion provides higher-order corrections to the normal approximation for a probability distribution. The classical proof of Edgeworth expansion is via characteristic functions.
As a powerful method for distributional approximations, Stein's method has also been used to prove Edgeworth expansion results. However, these results assume that either the test function is smooth (which excludes indicator functions of the half line) or that the random variables are continuous (which excludes random variables having only a continuous component). Thus, how to recover the classical Edgeworth expansion result using Stein's method has remained an open problem. In this paper, we develop Stein's method for two-term Edgeworth expansions in a general case.
Our approach involves repeated use of Stein equations, Stein identities via Stein kernels, and a replacement argument.

\medskip



\section{Introduction and main results}

Let $X_1, X_2, \dots$ be a sequence of independent and identically distributed (i.i.d.) random variables with mean 0 and variance 1. Let
$W_n=(X_1+\dots+X_n)/\sqrt{n}, n\geq 1$.
The central limit theorem states that as $n\to \infty$, the distribution of $W_n$ converges to the standard normal distribution. Under the finite third moment assumption, we have the following non-asymptotic error bound [\cite{berry1941accuracy,esseen1942liapunoff}]:
\ben{\label{eq:berry}
\sup_{x\in \mathbb{R}}|\P(W_n\leq x)-P(Z\leq x)|\leq \frac{\E |X_1|^3}{\sqrt{n}},
}
where $Z\sim N(0,1)$ is a standard normal random variable.

In general, the $1/\sqrt{n}$ rate in \eq{eq:berry} cannot be improved by considering the normal approximation of binomial distributions. However, under additional smoothness conditions, it is possible to devise asymptotic expansions incorporating higher-order correction terms for normal approximation. Such asymptotic expansions are called Edgeworth expansions. For example, if we assume that $X_1$ has a finite fourth moment and its characteristic function $\varphi(t)$ satisfies the Cram\'er condition:
\ben{\label{eq:cramer}
\limsup_{|t|\to \infty} |\varphi(t)|<1,
}
then we have
\ben{\label{eq:intro-0}
\sup_{x\in \mathbb{R}}\Big|\P(W_n\leq x)-P(Z\leq x)-\frac{\gamma}{6\sqrt{n}}\E \big[(Z^3-3Z) 1_{\{Z\leq x\}}\big]\Big|\leq \frac{C}{n},
}
where $\gamma=\E [X_1^3]$ and $C$ is a constant not depending on $n$.
The Cram\'er condition is satisfied if the distribution of $X_1$ has a nonzero and absolutely continuous component with respect to the Lebesgue measure.
We refer to \cite{petrov1975sums} and \cite{BR76} for classical Edgeworth expansion results with error bounds in both continuous and discrete cases and in multi-dimensions. As the condition \eq{eq:cramer} indicates, all of the classical proofs of Edgeworth expansion results are obtained via the characteristic function approach.

Stein's method was introduced in \cite{stein1972bound} and has become a main tool for distributional approximations (cf. \cite{ChGoSh11}).
Stein's method for asymptotic expansions was first studied by \cite{barbour1986asymptotic}. In the case of two-term Edgeworth expansion, he proved that (dropping the Cram\'er condition but assuming $h$ is second-order differentiable, see also \cite[Eq.(3.5)]{fang2019wasserstein})
\ben{\label{eq:intro-1}
\Big|\E h(W_n)-\E h(Z)-\frac{\gamma}{6\sqrt{n}}\E \big[(Z^3-3Z) h(Z)\big]\Big|\leq \frac{C}{n} \|h''\|_\infty,
}
where $\|h''\|_\infty:=\sup_{x\in \IR} |h''(x)|$.
See \cite{rinott2003edgeworth} for expansions of smooth test functions under local dependence, \cite{fang2020refined} for tail probability approximations, and \cite{braverman2022high} for non-normal approximations. 
\cite{kim2018edgeworth} considered asymptotic expansions for Gaussian functionals and \cite{fathi2021higher}  considered asymptotic expansions using higher-order Stein kernels and obtained Wasserstein bounds for normal approximations.

In summary, to obtain a rate faster than $n^{-1/2}$ in asymptotic expansions using Stein's method, studies have thus far assumed that (1) the test function is smooth, or (2) the random variable is continuous. Thus, how to recover the classical Edgeworth expansion result has remained an open problem. In this paper, we take the first step toward solving this problem by proving the following two-term Edgeworth expansion theorems using Stein's method.

\begin{theorem}[Continuous two-term Edgeworth expansion]\label{t1}
Let $n\geq 1$ and $X_1,\dots, X_n$ be i.i.d.\ with $\E X_1=0, \E X_1^2=1, \E X_1^3=\gamma, \E X_1^4<\infty$. Suppose $X_1$ has a compactly supported continuous component with density bounded away from 0. Let $W=\frac{1}{\sqrt{n}}\sum_{i=1}^n X_i$. Then, for all $h: \IR\to \IR$ with $\|h\|_\infty \leq 1$, we have
\ben{\label{eq:t1-0}
\Big|\E h(W)-\E h(Z)-\frac{\gamma}{6\sqrt{n}}\E \big[(Z^3-3Z) h(Z) \big]  \Big|\leq \frac{C}{n},
}
where $Z\sim N(0,1)$ and $C$ is a constant not depending on $n$. When $h(w)=1_{\{w\leq x\}}$, the correction term is $$\frac{\gamma}{6\sqrt{2\pi n}}\int_{-\infty}^x (y^3-3y) e^{-y^2/2} dy,$$ which is the same as in the classical Edgeworth expansion.
\end{theorem}

Although our condition is stronger than the Cram\'er condition, it is a natural condition because we do not use characteristic functions in the proof. It is possible but highly tedious to track the dependence of $C$ on the distribution of $X_1$, as it depends, for example, on the probability $p$ that $X_1$ takes the continuous component and on the density function of the continuous component. Therefore, we do not pursue this in this paper.

The difficulty in proving \eq{eq:intro-0} using Stein's method lies in the fact that using Taylor's expansion unavoidably results in an error of order $1/\sqrt{n}$.
To prove \eq{eq:t1-0}, we use an approach involving (a) repeated use of Stein equations, (b) Stein identities via Stein kernels, and (c) a replacement argument.
Asymptotic expansion using Stein kernels was considered by \cite{fathi2021higher}.
However, Fathi's approach does not appear to work for the expansion of distribution functions because (1) the existence of a Stein kernel requires the random variable to be continuous, and (2) his definition of higher-order Stein kernels involves higher-order derivatives of Stein equation solutions, which lack sufficient regularity.

Next, we consider the discrete case.

\begin{theorem}[Discrete two-term Edgeworth expansion]\label{t2}
Let $n\geq 1$ and $X_1,\dots, X_n$ be i.i.d., integer valued random variables with $\E X_1=\mu, \Var(X_1)=\sigma^2>0, \E (X_1-\mu)^3=\sigma^3 \gamma, \E X_1^4<\infty$. Suppose the support of $X_1$ is $\{s_0,s_1,s_2,\dots\}$ such that the greatest common divisor of $\{|s_i-s_0|,i\geq 1\}$ is 1. Let $W=\frac{1}{\sigma\sqrt{n}}\sum_{i=1}^n (X_i-\mu)$. Then, for all $h: \IR\to \IR$ such that $\|h\|_\infty \leq 1$ and $h(x)$ equals a constant (depending on $z$) for  $x\in (z-\frac{1}{2\sigma \sqrt{n}}, z+\frac{1}{2\sigma \sqrt{n}})$ for any $z$ in the support of $W$, we have
\ben{\label{eq:t2-0}
\Big|\E h(W)-\E h(Z)-\frac{\gamma}{6\sqrt{n}}\E \big[(Z^3-3Z) h(Z) \big]  \Big|\leq \frac{C}{n},
}
where $Z\sim N(0,1)$ and $C$ is a constant not depending on $n$. 
\end{theorem}

Note that the restriction on the test functions $h$ in \cref{t2} (i.e., the constants $\pm 1/2$) is necessary. As in continuity correction, changing them to other constants will increase the error rate to $1/\sqrt{n}$.

Although we focus on two-term Edgeworth expansion for the i.i.d.\ case in dimension one, our approach may work for multivariate approximations, for even higher-order expansions,  and for some dependent cases. See \cref{rem1,rem2,rem3} for related discussions. 

\section{Proofs}

In this section, we prove \cref{t1,t2}. For the sake of logical flow, we leave some standard computations to \cref{sec:app}.

\subsection{Stein kernel}

For both proofs of \cref{t1,t2}, we need the concept of Stein kernel.
The term ``Stein kernel'' first appeared in \cite{LeNoPe15}, but the concept goes back to \cite{stein1986approximate}, \cite{chatterjee2009fluctuations} and \cite{nourdin2009stein}. 

\begin{definition}[Stein kernel, \cite{saumard2019weighted}]\label{def:contkernel}
Let $W$ be a continuous random variable with mean $\mu$, a connected support $(a,b)$, and positive density $p$ in $(a,b)$.
The function $\tau:\IR\to \IR$ is the Stein kernel for $W$ if 
\begin{equation} \label{eq:tau1}
    \E[(W-\mu)f(W)] = \E[\tau(W)f'(W)]
\end{equation}
for any differentiable function $f:\IR\to \IR$ such that the above expectations exist and $\tau p f\vert_a^b=0$.
\end{definition}

It is known that (cf. \cite{saumard2019weighted}) $\tau$ can be taken as
\ben{\label{eq:saumard}
\tau(x)=\frac{1}{p(x)}\int_x^\infty (y-\mu) p(y) dy,\quad x\in (a,b).
}
If, in addition, $\Var(W)=\sigma^2$, then by choosing $f(w)=w$ in \eq{eq:tau1}, we find that $\E \tau(W)=\sigma^2$. 

\begin{remark}\label{rem1}
The existence of a Stein kernel in multi-dimensions is more complicated (cf. \cite{CoFaPa19,Fa19,fathi2021higher}). \cite{CoFaPa19} proved that Stein kernel exists if a probability measure on $\IR^d$ satisfies a Poincar\'e inequality. Therefore, it may be possible to use our approach for multivariate Edgeworth expansions, assuming that the summand $X_i$ has a component uniformly distributed on a Euclidean ball in $\IR^d$.
\end{remark}




\subsection{Proof of Theorem \ref{t1}}

In both proofs of \cref{t1,t2}, we will need the following Edgeworth expansion result for sums of i.i.d.\ random variables having Stein kernels.

\begin{lemma}[Stein kernel bound]\label{l1}
Let $n\geq 1$ and $X_1,\dots, X_n$ be i.i.d.\ with $\E X_1=0, \E X_1^2=1, \E X_1^3=\gamma, \E X_1^4<\infty$. Suppose $X_1$ has a Stein kernel $\tau_1$ such that $\E[\tau_1^2(X_1)]<\infty$. Let $W=\frac{1}{\sqrt{n}}\sum_{i=1}^n X_i$. Then, for all $h: \IR\to \IR$ with $\|h\|_\infty \leq 1$, we have
\ben{\label{eq:l1-0}
\Big|\E h(W)-\E h(Z)-\frac{\gamma}{6\sqrt{n}}\E [(Z^3-3Z) h(Z)]  \Big|\leq \frac{C}{n}\Big(\E [X_1^4]+\E[\tau_1^2(X_1)]\Big),
}
where $Z\sim N(0,1)$ and $C$ is a universal constant.
\end{lemma}

\begin{proof}[Proof of Lemma \ref{l1}]
We use $C$ to denote universal constants in this proof. Its value may differ from line to line.
Let $f$ be the bounded solution to the Stein equation for $Z\sim N(0,1)$:
\ben{\label{eq:l1-1}
f'(w)-w f(w)=h(w)-\E h(Z).
}
It is known that (cf. \cite[Lemma 2.4]{ChGoSh11})
\ben{\label{eq:l1-13}
f(w)=e^{w^2/2} \int_{-\infty}^w (h(x)-\E h(Z)) e^{-x^2/2}dx
}
and
\ben{\label{eq:l1-2}
\|f\|_\infty \leq \sqrt{\pi/2}\|h-\E h(Z)\|_\infty,\quad \|f'\|_\infty \leq 2\|h-\E h(Z)\|_\infty.
}
Because $X_1$ has Stein kernel $\tau_1$, we have 
\be{
\E[Wf(W)]=\frac{1}{\sqrt{n}}\sum_{i=1}^n \E[X_i f(W)]=\frac{1}{n} \sum_{i=1}^n\E [ \tau_1(X_i) f'(W)].
}
This implies
\ben{\label{eq:l1-3}
\E h(W)-\E h(Z)=\E f'(W)-\E [Wf(W)]=\E [(1-\tau)f'(W)],
}
where
\be{
\tau:=\frac{1}{n}\sum_{i=1}^n\tau_1(X_i).
}
Therefore, using $\E\tau=\Var(W)=1$, we have
\ben{\label{eq:l1-8}
|\E h(W)-\E h(Z)|\leq \frac{2}{\sqrt{n}}\|h-\E h(Z)\|_\infty \sqrt{\E[\tau_1^2(X_i)]}. 
}
This bound is well known in the literature (cf. \cite[Lemma 5.3]{chatterjee2009fluctuations}). To obtain higher-order expansions, motivated by \cite[Eq.(33)]{kim2018edgeworth}, we make repeated use of the Stein equation as follows. 
Recall that $f$ is the bounded solution to \eq{eq:l1-1}.
Let $g$ be the bounded solution to the Stein equation
\ben{\label{eq:l1-20}
g'(w)-wg(w)=f'(w)-\E f'(Z).
}
From \eq{eq:l1-2} and the condition that $\|h\|_\infty\leq 1$, we have 
\be{
\|g\|_\infty \leq C,\quad \|g'\|_\infty \leq C.
}
Continuing from \eq{eq:l1-3} and using $\E \tau=1$ and \eq{eq:l1-20}, we have
\ben{\label{eq:l1-4}
\E h(W)-\E h(Z)=\E[(1-\tau)(f'(W)-\E f'(Z))]=\E[(1-\tau)(g'(W)-Wg(W))].
}
Denote $\tau^{(i)}:=\tau-\tau_1(X_i)/n$. We have
\besn{\label{eq:l1-5}
&\quad\E [(1-\tau)Wg(W)]\\
&=\frac{1}{\sqrt{n}} \sum_{i=1}^n \E [ (\frac{n-1}{n}-\tau^{(i)})X_i g(W)] + \frac{1}{\sqrt{n}} \sum_{i=1}^n \E [ (\frac{1}{n}-\frac{\tau_1(X_i)}{n})X_i g(W)]\\
&=\frac{1}{n} \sum_{i=1}^n \E [ (\frac{n-1}{n}-\tau^{(i)})\tau_1(X_i) g'(W)] + \frac{1}{\sqrt{n}} \sum_{i=1}^n \E [ (\frac{1}{n}-\frac{\tau_1(X_i)}{n})X_i g(W)]\\
&=\E [(1-\tau)\tau g'(W)]-\frac{1}{n} \sum_{i=1}^n \E [ (\frac{1}{n}-\frac{\tau_1(X_i)}{n})\tau_1(X_i) g'(W)] + \frac{1}{\sqrt{n}} \sum_{i=1}^n \E [ (\frac{1}{n}-\frac{\tau_1(X_i)}{n})X_i g(W)],
}
where we used the definition of Stein kernel in the second equation. 
From the boundedness of $g'$ and $\E \tau_1(X_i)=\Var(X_i)=1$, we have
\ben{\label{eq:l1-6}
\Big| \frac{1}{n} \sum_{i=1}^n \E [ (\frac{1}{n}-\frac{\tau_1(X_i)}{n})\tau_1(X_i) g'(W)] \Big|\leq \frac{C}{n}\E[\tau_1^2(X_1)].
}
From the definition of Stein kernel, we have $\tau_1\geq 0$, $\E[X_i \tau_1(X_i)]=\E[X_i^3]/2=\gamma/2$ and $\E[X_i^2 \tau_1(X_i)]=\E[X_i^4]/3$. Therefore, using independence, we have
\besn{\label{eq:l1-7}
&\Big| \frac{1}{\sqrt{n}} \sum_{i=1}^n \E [ (\frac{1}{n}-\frac{\tau_1(X_i)}{n})X_i g(W)]+\frac{\gamma}{2\sqrt{n}}\E g(W)\Big|\\
=& \Big|\frac{1}{\sqrt{n}} \sum_{i=1}^n \E [ \frac{X_i-X_i\tau_1(X_i)+\gamma/2}{n} (g(W)-g(W^{(i)}))] \Big|\\
\leq &\frac{C}{n} \E [X_1^4],
}
where $W^{(i)}:=W-X_i/\sqrt{n}$ and we used the boundedness of $g'$ and the Cauchy-Schwarz inequality in the last inequality. From \eq{eq:l1-4}--\eq{eq:l1-7}, we have
\ben{\label{eq:l1-9}
\Big|\E h(W)-\E h(Z)- \frac{\gamma}{2\sqrt{n}} \E g(W)-\E[(1-\tau)^2 g'(W)]\Big|\leq \frac{C}{n} \Big(\E [X_1^4]+\E[\tau_1^2(X_1)]\Big).
}
Using the boundedness of $g$, \eq{eq:l1-8}, $ab\leq (a^2+b^2)/2$ and the Cauchy-Schwarz inequality, we have
\ben{\label{eq:l1-10}
\Big|\frac{\gamma}{2\sqrt{n}} \E [g(W)-g(Z)] \Big|\leq \frac{C}{n} \Big(\E [X_1^4]+\E[\tau_1^2(X_1)]\Big).
}
Using the boundedness of $g'$ and $\E \tau=1$, we have
\ben{\label{eq:l1-11}
\Big| \E[(1-\tau)^2 g'(W)]\Big|\leq \frac{C}{n} \E[\tau_1^2(X_1)].
}
Finally, \eq{eq:l1-0} follows from \eq{eq:l1-9}--\eq{eq:l1-11} and 
\ben{\label{eq:l1-12}
\E g(Z)=\frac{1}{3}\E [(Z^3-3Z) h(Z)].
}
The last equation is proved in \cref{sec:3.1} by a standard computation.
\end{proof}

\begin{remark}\label{rem2}
As indicated by \cite[Eq.(37)]{kim2018edgeworth}, further repetitions of arguments in the above proof can be made to obtain even higher-order expansions. 
We do not pursue these in this paper as they unavoidably require tedious notation.
\end{remark}

To prove Theorem \ref{t1}, we divide $X_1$ into two components. We use Lemma \ref{l1} to deal with the continuous component. It then becomes a problem of asymptotic expansion of the expectation of a smooth test function of the remaining components. Finally, we use \cref{l1} again to approximate a sum of Gaussian mixtures by a normal distribution.

\begin{proof}[Proof of Theorem \ref{t1}]
In this proof, we use $C$ to denote positive constants that depend only on the distribution of $X_1$ and may differ from line to line.
We use $O(1)$ to denote a quantity that is bounded by $C$ in absolute value.
Without loss of generality, we assume that with probability $p>0$, $X_1=U_1$ and
\ben{\label{eq:t1-1}
\E U_1=0, \  U_1\  \text{is continuous, compactly supported and has density bounded away from 0}.
}
This centering can be achieved by a finite convolution that does not affect the error rate (see details in \cref{sec:3.2}). 
We remark that assuming $\E U_1=0$ is only for convenience, as otherwise the term $III$ below would result in an additional non-negligible term.
Let $\tilde U_1$ be the other component of $X_1$, which may be discrete or even singular. That is, 
\begin{equation*}
X_1=
\begin{cases}
U_1,& \text{with probability}\  p,\\
\tilde U_1, & \text{with probability}\  1-p.
\end{cases}
\end{equation*}
Suppose
\be{
\E U_1=0,\  \E U_1^2=\sigma_1^2,\  \E U_1^3=\gamma_1,\E \tilde U_1=0, \ \E \tilde U_1^2=\sigma_2^2,\ \E\tilde U_1^3=\gamma_2.
}
Assume $\sigma_2^2>0$; otherwise, \cref{t1} follows directly from \cref{l1}.
(Although we do not do it here, an inspection of \cref{sec:3.2} shows that we can further assume $\E U_1^3=0$. This would slightly simplify the following proof by making $\gamma_1=0$.)
Note that
\be{
\E X_1=0,\ \E X_1^2=1=p\sigma_1^2+(1-p)\sigma_2^2,\ \E X_1^3=\gamma=p \gamma_1^3+(1-p) \gamma_2^3.
}
Let 
\begin{equation}\label{eq:t1-2}
Y_1=
\begin{cases}
V_1\sim N(0, \sigma_1^2),& \text{with probability}\  p,\\
\tilde V_1\sim N(0, \sigma_2^2), & \text{with probability}\  1-p.
\end{cases}
\end{equation}
Let $U_1, \dots, U_n, \tilde U_1, \dots, \tilde U_n, V_1,\dots, V_n, \tilde V_1,\dots, \tilde V_n$ be jointly independent, where $U_1, \dots, U_n$ are identically distributed (the same applies for $\tilde U$'s, $V$'s and $\tilde V$'s, respectively).
Let $Z, \tilde Z\sim N(0,1)$ be independent standard normal variables and independent of everything else.

Let $L\sim Bin(n,p)$ be a binomial random variable independent of everything else. Then, $W$ has the same distribution as $(U_1+\dots+U_L+\tilde U_{L+1}+\dots +\tilde U_{n})/\sqrt{n}$.
We write $\E h(W)-\E h(Z)$ into three terms as follows:
\besn{\label{eq:t1-5}
&\E h(W)-\E h(Z)\\
=& \E h\Big(\frac{U_1+\dots+U_L+\tilde U_{L+1}+\dots + \tilde U_{n}}{\sqrt{n}}\Big)-\E h\Big(\frac{V_1+\dots+V_L+\tilde U_{L+1}+\dots + \tilde U_{n}}{\sqrt{n}}\Big)\\
&+\E h\Big(\frac{V_1+\dots+V_L+\tilde U_{L+1}+\dots +\tilde U_{n}}{\sqrt{n}}\Big)-\E h\Big(\frac{V_1+\dots+V_L+\tilde V_{L+1}+\dots + \tilde V_{n}}{\sqrt{n}}\Big)\\
&+\E h\Big(\frac{V_1+\dots+V_L+\tilde V_{L+1}+\dots + \tilde V_{n}}{\sqrt{n}}\Big)-\E h(Z)\\
=&I+II+III.
}
Note that $(V_1+\dots+V_L+\tilde V_{L+1}+\dots +\tilde V_{n})/\sqrt{n}$ has the same distribution as $(Y_1+\dots+Y_n)/\sqrt{n}$, where $Y_1,\dots, Y_n$ are i.i.d.\ with distribution \eq{eq:t1-2}.
From the Stein kernel bound \eq{eq:l1-0},
$\E Y_1^3=0$, $\E Y_1^4<\infty$ and the boundedness of its Stein kernel (this can be easily checked from the expression \eq{eq:saumard}, see \cref{lem:mixture}), we have
\ben{\label{eq:t1-4}
|III|\leq \frac{C}{n}.
}
Given $L=l$ such that $|l-np|\leq np/2$, we have, by conditioning on the  $\tilde{U}$'s and using the Stein kernel bound \eq{eq:l1-0} and the condition \eq{eq:t1-1},
\besn{\label{eq:t1-3}
&\E h\Big(\frac{U_1+\dots+U_l+\tilde U_{l+1}+\dots + \tilde  U_{n}}{\sqrt{n}}\Big)-\E h\Big(\frac{V_1+\dots+V_l+\tilde U_{l+1}+\dots + \tilde U_{n}}{\sqrt{n}}\Big)\\
=& \E h\Big(\frac{U_1+\dots+U_l}{\sigma_1\sqrt{l}}\cdot\frac{\sigma_1\sqrt{l}}{\sqrt{n}} +\frac{\tilde U_{l+1}+\dots+ \tilde U_{n}}{\sqrt{n}}\Big)-\E h\Big(Z\cdot\frac{\sigma_1\sqrt{l}}{\sqrt{n}}+\frac{\tilde U_{l+1}+\dots + \tilde U_{n}}{\sqrt{n}}\Big)   \\
=& \frac{\gamma_1}{6\sigma_1^3\sqrt{l}} \E\Big[ (Z^3-3Z) h(Z\cdot \frac{\sigma_1\sqrt{l}}{\sqrt{n}}+\frac{\tilde U_{l+1}+\dots +\tilde U_{n}}{\sqrt{n}}) \Big]+O(\frac{1}{n}).
}
By regarding the last expectation as an expectation of a smooth test function (because of convolution with normal) of the $\tilde U$'s, we have (see \cref{sec:3.3})
\ben{\label{eq:t1-7}
\eq{eq:t1-3}=\frac{\gamma_1}{6\sigma_1^3\sqrt{l}} \E\Big[ (Z^3-3Z) h(Z\cdot \frac{\sigma_1\sqrt{l}}{\sqrt{n}}+\tilde Z\cdot \frac{\sigma_2\sqrt{n-l}}{\sqrt{n}}) \Big]+O(\frac{1}{n}).
}
Because $L\sim Bin(n,p)$, the event $\{|L-np|> np/2\}$
has probability $O(1/n)$.
Therefore, 
\be{
I=\frac{1}{6\sqrt{n}} \E\Big[ \frac{\gamma_1 \sqrt{n}}{\sigma_2^3\sqrt{L} }(Z^3-3Z) h\big(Z\cdot \frac{\sigma_1\sqrt{L}}{\sqrt{n}}+\tilde Z\cdot \frac{\sigma_2\sqrt{n-L}}{\sqrt{n}}\big) 1_{\{|L-np|\leq np/2\}} \Big]+O(\frac{1}{n}).
}
By a similar reasoning (conditioning on $L$ and using the smooth function expansion for the $\tilde U$'s by applying either \eq{eq:intro-1} or a Lindeberg swapping argument), we have
\be{
II=\frac{1}{6\sqrt{n}} \E\Big[ \frac{\gamma_2 \sqrt{n}}{\sigma_1^3\sqrt{n-L} }(\tilde Z^3-3\tilde Z) h\big(Z\cdot \frac{\sigma_1\sqrt{L}}{\sqrt{n}}+\tilde Z\cdot \frac{\sigma_2\sqrt{n-L}}{\sqrt{n}}\big) 1_{\{|L-np|\leq np/2\}} \Big]+O(\frac{1}{n}).
}
Denote
\be{
\sigma_L^2:=\frac{L \sigma_1^2 + (n-L) \sigma_2^2}{n},\quad \gamma_L:= \frac{L\gamma_1+(n-L) \gamma_2}{n}.
}
Using Gaussian integration by parts and then combining the two independent Gaussian variables (and by approximating $h$ using arbitrarily close smooth functions in the intermediate step), we have
\besn{\label{eq:t1-60}
I+II=&-\frac{1}{6\sqrt{n}} \E\Big[\gamma_L h'''(\sigma_L Z)1_{\{|L-np|\leq np/2\}}\Big]+O(\frac{1}{n})\\
=& \frac{1}{6\sqrt{n}} \E \Big[ \frac{\gamma_L}{\sigma_L^3} (Z^3-3Z) h(\sigma_L Z)1_{\{|L-np|\leq np/2\}} \Big] +O(\frac{1}{n})\\
=& \frac{1}{6\sqrt{n}} \E \Big[ \frac{\gamma_L}{\sigma_L^3}h_{2}(\sigma_{L}) \Big] +O(\frac{1}{n}),
}
where $h_{2}(x)=\E[(Z^3-3Z) h(x Z)]$ and we removed the indicator in the last equation because the event $\{|L-np|\leq np/2\}$ occurs with overwhelming probability. Using Gaussian integration by parts, we have
$\vert h_{2}(x) \vert \leq C$ and 
$\lvert h_{2}'(x) \rvert \leq \frac{C}{x}$ for $x>0$ (cf. \cref{eq:lem5}). 
According to the definition of $\sigma_{L}^{2}$, we have $\sigma_{L}^{2}\geq \min
\{\sigma_{1}^{2},\sigma_{2}^{2}\}>0$. Moreover, 
$$ \sigma_{L}^2-1=(\sigma_{1}^{2}-\sigma_{2}^{2})\frac{L-np}{n},$$
$$ \gamma_{L}-
\gamma=(\gamma_{1}-\gamma_{2})\frac{L-np}{n}.$$
By applying Taylor's expansion to $\gamma_L h_{2}(\sigma_{L})/\sigma_{L}^{3}$ with respect to
$\sigma_{L}$ ($\gamma_L$ resp.) at point $1$ ($\gamma$ resp.) and then taking expectation with respect to $L$, we obtain
\besn{\label{eq:t1-6}
 I+II
=\frac{\gamma}{6\sqrt{n}} \E \Big[  (Z^3-3Z) h( Z) \Big] +O(\frac{1}{n}).
}
Combining \eq{eq:t1-5}, \eq{eq:t1-4} and \eq{eq:t1-6}, we obtain \eq{eq:t1-0}.
\end{proof}

\begin{remark}\label{rem3}
Our approach may also work for some dependent cases. For example, following the proof of \cref{l1}, we may obtain an Edgeworth expansion result in normal approximation of multilinear forms of independent random variables having Stein kernels. Then, combining with the replacement argument in the proof of \cref{t1}, we may deal with multilinear forms of independent random variables having continuous components.
\end{remark}

\subsection{Proof of Theorem \ref{t2}}





In this proof, we use $C$ to denote positive constants that depend only on the distribution of $X_1$ and may differ from line to line.
We use $O(1)$ to denote a quantity that is bounded by $C$ in absolute value.
We divide the proof into four steps following the approach used in the proof of \cref{t1}. Note that although we deal with a discrete case here, we still need the Stein kernel bound in \cref{l1} for approximating a sum of Gaussian mixtures.

\medskip

\textbf{Step 1.}
As in the proof of \cref{t1}, we assume without loss of generality that $X_1$ has a Bernoulli $Ber(1/2)$ component. This is possible for $X_1+\dots+X_m-z$ with a positive integer $m$ and an integer $z$ because of the assumption on the support of $X_1$ (see \cref{sec:GCD}). Such finite grouping does not affect the error rate following similar arguments as \cref{eq:1007}--\cref{eq:10010} in \cref{sec:3.2}. Under this assumption, we have
\begin{equation*}
X_1=
\begin{cases}
U_1\sim Ber(1/2),& \text{with probability}\  p,\\
\tilde U_1\in\IZ, & \text{with probability}\  1-p,
\end{cases}
\end{equation*}
and
\be{
\E U_1=\frac{1}{2},\  \Var (U_1)=\frac{1}{4},\  \E (U_1-\frac{1}{2})^3=0,
}
\be{
\E \tilde U_1=:\mu_2, \ \Var (\tilde U_1)=:\sigma_2^2,\ \E (\tilde U_1-\mu_2)^3=:\gamma_2.
}
\begin{equation}
    \E (X_{1}-\mu)^{3}= p\left(\frac{3}{4}
    (\frac{1}{2}-\mu)+(\frac{1}{2}-\mu)^3\right)
    +(1-p)(\gamma_{2}+3\sigma_2^{2}(\mu_2-\mu)+(\mu_2-\mu)^3). 
    \label{eq:10042}
\end{equation}
Note that if $\sigma_2^2=0$, then \cref{t2} follows directly from \cref{lem:binomial} below. Therefore, we assume that $\sigma_2^2>0$ in the following.
We remark that, unlike \eq{eq:t1-1}, it appears impossible to center $U_1$ in the discrete case. This makes the current proof slightly more technical than that of \cref{t1}.
Let 
\begin{equation}\label{eq:t2-1}
Y_1=
\begin{cases}
V_1\sim N(\frac{1}{2}, \frac{1}{4}),& \text{with probability}\  p,\\
\tilde V_1\sim N(\mu_2, \sigma_2^2), & \text{with probability}\  1-p.
\end{cases}
\end{equation}
Let $U_1, \dots, U_n, \tilde U_1, \dots, \tilde U_n, V_1,\dots, V_n, \tilde V_1,\dots, \tilde V_n$ be jointly independent, where $U_1, \dots, U_n$ are identically distributed (the same applies for $\tilde U$'s, $V$'s and $\tilde V$'s, respectively).
Let $Z, \tilde Z\sim N(0,1)$ be independent standard normal variables and independent of everything else.

Let $L\sim Bin(n,p)$ be independent of everything else. Then $W$ has the same distribution as $(U_1+\dots+U_L+\tilde U_{L+1}+\dots +\tilde U_{n}-n\mu)/(\sigma\sqrt{n})$.
We write $\E h(W)-\E h(Z)$ into three terms as follows:
\besn{\label{eq:t2-2}
&\E h(W)-\E h(Z)\\
=& \E h\Big(\frac{U_1+\dots+U_L+\tilde U_{L+1}+\dots + \tilde U_{n}-n\mu}{\sigma\sqrt{n}}\Big)-\E h\Big(\frac{V_1+\dots+V_L+\tilde U_{L+1}+\dots + \tilde U_{n}-n\mu}{\sigma\sqrt{n}}\Big)\\
&+\E h\Big(\frac{V_1+\dots+V_L+\tilde U_{L+1}+\dots +\tilde U_{n}-n\mu}{\sigma\sqrt{n}}\Big)-\E h\Big(\frac{V_1+\dots+V_L+\tilde V_{L+1}+\dots + \tilde V_{n}-n\mu}{\sigma\sqrt{n}}\Big)\\
&+\E h\Big(\frac{V_1+\dots+V_L+\tilde V_{L+1}+\dots + \tilde V_{n}-n\mu}{\sigma\sqrt{n}}\Big)-\E h(Z)\\
=&I+II+III.
}
Note that $(V_1+\dots+V_L+\tilde V_{L+1}+\dots +\tilde V_{n}-n\mu)/(\sigma\sqrt{n})$ has the same distribution as $(Y_1+\dots+Y_n-n\mu)/(\sigma\sqrt{n})$, where $Y_1,\dots, Y_n$ are i.i.d.\ with distribution \eq{eq:t2-1}.
From the Stein kernel bound \eq{eq:l1-0}, $\E Y_1^4<\infty$ and the boundedness of its Stein kernel (see \cref{lem:mixture}),
we have
\ben{\label{eq:10040}{{~~}}
III= \frac{\E(Y_1-\mu)^3}{6 \sqrt{n} \sigma^3} \E [(Z^3-3Z)h(Z)]+O(\frac{1}{n}).
}
It is straightward to compute that
\ben{\label{eq:t2-5}
\E (Y_1-\mu)^3= p\left(\frac{3}{4} (\frac{1}{2}-\mu)+(\frac{1}{2}-\mu)^3\right) +(1-p)(3\sigma_2^2(\mu_2-\mu)+(\mu_2-\mu)^3).
}
\medskip

\textbf{Step 2.} To deal with $I$, we use the following lemma, which is proved in \cref{sec:3.4}.

\begin{lemma}\label{lem:binomial}
Let $l\geq 1$, $S\sim Bin(l,1/2)$ and $Z\sim N(0,1)$. Let $h:\IR\to \IR$ be such that $|h(x)|\leq 1$ and $h(x)$ equals a constant in $(z-1/2, z+1/2)$ for any $z\in \IZ$. Then we have
\be{
\E h(S)-\E h\left(\frac{l}{2}+Z\sqrt{\frac{l}{4}}\right)=O\left(\frac{1}{l}\right).
}
\end{lemma}

Given $L=l$ such that $|l-np|\leq np/2$, we have, by conditioning on the $\tilde U$'s and using \cref{lem:binomial} and the condition on $h$ in \cref{t2},
\bes{
&\E h\Big(\frac{U_1+\dots+U_l+\tilde U_{l+1}+\dots + \tilde  U_{n}-n\mu}{\sigma\sqrt{n}}\Big)-\E h\Big(\frac{V_1+\dots+V_l+\tilde U_{l+1}+\dots + \tilde U_{n}-n\mu}{\sigma\sqrt{n}}\Big)\\
=& O(\frac{1}{n}).
}
Because $L\sim Bin(n,p)$, the event $\{|L-np|> np/2\}$
has probability $O(1/n)$.
Therefore, 
\ben{\label{eq:t2-3}
I=O(\frac{1}{n}).
}

\medskip

\textbf{Step 3.} 
By a similar reasoning as in estimating $II$ in the proof of \cref{t1} (conditioning on $L$ and using the smooth function expansion for the $\tilde U$'s by either \eq{eq:intro-1} or a Lindeberg swapping argument), we have
\ben{\label{eq:10031}
    II=\frac{\gamma_{2} }{6 \sigma_{2}^{3}}
    \E\Big[\frac{1}{\sqrt{n-L}}
    (\tilde{Z}^{3}-3\tilde{Z})h\left(\frac{\sqrt{L}}{2 \sigma
            \sqrt{n}}Z+\frac{\sigma_{2}\sqrt{n-L}}{\sigma \sqrt{n}}\tilde{Z}+\mu_{L}\right)
            1_{ \{|L-np|\leq np/2\}}\Big]+O(\frac{1}{n}),
}
where 
\ben{\label{eq:t2-6}
\mu_{L}=\frac{L/2-\mu L}{\sigma\sqrt{n}}+\frac{(n-L)\mu_{2}-(n-L)\mu}{\sigma \sqrt{n}}.
}
Given $L=l$ such that $|l-np|\leq np/2$, let $$h_{1}(x)=\E\Big[
     (\tilde{Z}^{3}-3\tilde{Z})h\left(x+\frac{\sigma_{2}\sqrt{n-l}}{\sigma
         \sqrt{n}}\tilde{Z}+\mu_{l}\right)\Big],$$
and 
$$ h_{2}(r)= \E h_{1}(r Z).$$
         Because $ \lVert h_{1} \rVert_{\infty}\leq C$, from \eq{eq:lem5} and Taylor's expansion, we have 
         \begin{equation}
             \begin{aligned}
                 h_{2}\left(\frac{\sqrt{l}}{2 \sigma
                     \sqrt{n}}\right)=h_{2}\left( \frac{\sqrt{p}}{2\sigma} \right) +
                     O(1) \frac{1}{\sqrt{p}}\frac{l-np}{n}.
             \end{aligned}
             \label{eq:10032}
         \end{equation}
         Plugging \cref{eq:10032} into \cref{eq:10031}, taking expectation with respect to $L$, and using $\E|L-np|=O(\sqrt{n})$, we obtain 
         \begin{equation}
             \begin{aligned}
                II= \frac{\gamma_2 }{6 \sigma_{2}^{3}}
    \E\Big[\frac{1}{\sqrt{n-L}}
    (\tilde{Z}^{3}-3\tilde{Z})h\left(\frac{\sqrt{p}}{2 \sigma
            }Z+\frac{\sigma_{2}\sqrt{n-L}}{\sigma \sqrt{n}}\tilde{Z}+\mu_{L}\right)
            1_{ \{|L-np|\leq np/2\}}\Big]+O(\frac{1}{n}).
             \end{aligned}
             \label{eq:10033}
         \end{equation}
         
         Let $h_{3}(x)=\E h(\frac{\sqrt{p}}{2\sigma}Z+x)$. By the Gaussian integration by parts formula, we have
         \begin{equation}
             \begin{aligned}
                 II= -\frac{\gamma_2 }{6 \sqrt{n} \sigma^{3}}
                 \E\Big[\frac{n-L}{n}
                 h_{3}'''\left(\frac{\sigma_{2}\sqrt{n-L}}{\sigma \sqrt{n}}\tilde{Z}+\mu_{L}\right)
            1_{ \{|L-np|\leq np/2\}}\Big]+O(\frac{1}{n}).
             \end{aligned}
             \label{eq:10034}
         \end{equation}
         Similar to \cref{eq:10033}, we obtain
         \begin{equation}
             \begin{aligned}
                 II=& -\frac{\gamma_2 }{6 \sqrt{n} \sigma^{3}}
                 \E\Big[\frac{n-L}{n}
                 h_{3}'''\left(\frac{\sigma_{2}\sqrt{n-np}}{\sigma \sqrt{n}}\tilde{Z}+\mu_{L}\right)
            1_{ \{|L-np|\leq np/2\}}\Big]+O(\frac{1}{n})
               \\=& -\frac{\gamma_2 }{6 \sqrt{n} \sigma^{3}}
                 \E\Big[\frac{n-L}{n}
                 h_{3}'''\left(\frac{\sigma_{2}\sqrt{n-np}}{\sigma \sqrt{n}}\tilde{Z}+\mu_{L}\right)\Big]
           +O(\frac{1}{n})
               \\=& -\frac{\gamma_2(1-p) }{6 \sqrt{n} \sigma^{3}}
                 \E\Big[h_{3}'''\left(\frac{\sigma_{2}\sqrt{n-np}}{\sigma
                 \sqrt{n}}\tilde{Z}+\mu_{L}\right)\Big]\\
                  &-\frac{\gamma_2 }{6 \sqrt{n} \sigma^{3}}
                  \E\Big[\frac{np-L}{n}
                 h_{3}'''\left(\frac{\sigma_{2}\sqrt{n-np}}{\sigma
                 \sqrt{n}}\tilde{Z}+\mu_{L}\right)\Big]
           +O(\frac{1}{n})
\\=& -\frac{\gamma_2(1-p) }{6 \sqrt{n} \sigma^{3}}
                 \E
                 \Big[ h_{3}'''\left(\frac{\sigma_{2}\sqrt{n-np}}{\sigma
                 \sqrt{n}}\tilde{Z}+\mu_{L}\right)\Big]
           +O(\frac{1}{n}),
             \end{aligned}
             \label{eq:10035}
         \end{equation}
        where we use the fact that the event $\{|L-np|> np/2\}$
        has probability $O(1/n)$ to drop the indicator in the second equality, and we used the boundedness of $h_3'''$ (cf. \eq{eq:lem1}) and the fact that $\E|L-np|=O(\sqrt{n})$ to drop the second term in the last equation.

        From \eq{eq:t2-6} and $\mu=p/2+(1-p)\mu_2$, we have $$\mu_{L}= \left( \frac{1}{2}-\mu_{2}
        \right)\frac{L-np}{\sigma \sqrt{n}}=\left( \frac{1}{2}-\mu_{2}
    \right)\frac{\sqrt{p(1-p)}}{\sigma}\frac{L-np}{ \sqrt{np(1-p)}}.$$ 
    Note that $\mu_{L}$ converges in distribution to $N(0,\left( {1}/{2}-\mu_{2}
    \right)^2{p(1-p)}/{\sigma^2})$ as $n\to \infty$.
    Since $h_{3}$ has bounded derivatives (cf. \eq{eq:lem1}), we have, by Lindeberg's swapping argument, 
    \begin{equation}
        \begin{aligned}
            II=-\frac{\gamma_2(1-p) }{6 \sqrt{n} \sigma^{3}}
                 \E
                 h_{3}'''\left(\frac{\sigma_{2}\sqrt{n-np}}{\sigma
                 \sqrt{n}}\tilde{Z}+\left( \frac{1}{2}-\mu_{2}
    \right)\frac{\sqrt{p(1-p)}}{\sigma} \hat{Z}\right)
           +O(\frac{1}{n}),
        \end{aligned}
        \label{eq:10036}
    \end{equation}
    where $\hat{Z}\sim N(0,1)$ and $\hat Z$ is independent of everything else.
    Because $Z,\tilde{Z},\hat{Z}$ are i.i.d.\  $N(0,1)$, by a straightforward computation, we have 
    \begin{equation*}
        \begin{aligned}
            \Var\left(\frac{\sqrt{p}}{2\sigma}Z+\frac{\sigma_{2}\sqrt{n-np}}{\sigma
                 \sqrt{n}}\tilde{Z}+\left( \frac{1}{2}-\mu_{2}
             \right)\frac{\sqrt{p(1-p)}}{\sigma} \hat{Z}\right)=1.
        \end{aligned}
    \end{equation*}
    Combining the three independent Gaussian variables and using Gaussian integration by parts (and approximating $h$ by
    arbitrarily close smooth functions in the intermediate step), we have
    \begin{equation}
        \begin{aligned}
            II=&-\frac{\gamma_2(1-p) }{6 \sqrt{n} \sigma^{3}}
                 \E
                 h'''\left(\frac{\sqrt{p}}{2\sigma}Z+\frac{\sigma_{2}\sqrt{n-np}}{\sigma
                 \sqrt{n}}\tilde{Z}+\left( \frac{1}{2}-\mu_{2}
    \right)\frac{\sqrt{p(1-p)}}{\sigma} \hat{Z}\right)
           +O(\frac{1}{n})
            \\=&-\frac{\gamma_2(1-p) }{6 \sqrt{n} \sigma^{3}}
                 \E h'''\left(Z\right)
           +O(\frac{1}{n})
\\=&\frac{\gamma_2(1-p) }{6 \sqrt{n} \sigma^{3}}
\E[(Z^{3}-3Z) h\left(Z\right)]
           +O(\frac{1}{n}).
        \end{aligned}
        \label{eq:10038}
    \end{equation}

\medskip

\textbf{Step 4.} 
From \eq{eq:10042} and \eq{eq:t2-5}, we have
\begin{equation}
    \E(Y_{1}-\mu)^3+ (1-p)\gamma_{2}=\E (X_{1}-\mu)^{3}=\sigma^{3} \gamma.
    \label{eq:10041}
\end{equation}
From \eq{eq:10038}, \eq{eq:10040} and \eq{eq:10041}, we have
\besn{\label{eq:t2-4}
 II+III
=\frac{\gamma}{6\sqrt{n}} \E \Big[  (Z^3-3Z) h( Z) \Big] +O(\frac{1}{n}).
}
Combining \cref{eq:t2-2,eq:t2-3,eq:t2-4}, we obtain \eq{eq:t2-0}.




\section{Appendix}\label{sec:app}

\subsection{Proof of (\ref{eq:l1-12})}\label{sec:3.1}
From \eq{eq:l1-13}, the solution to \cref{eq:l1-20} is 
\be{g(w)=e^{w^2/2} \int_{-\infty}^w (f'(x)-\E f'(Z)) e^{-x^2/2}dx.}
By the integration by parts formula,
\begin{equation}
    \begin{aligned}
        \E g(Z) =& \frac{1}{\sqrt{2\pi}}\int_{w\in \mathbb{R}}^{ } \int_{-\infty}^w (f'(x)-\E f'(Z))
        e^{-x^2/2}dxdw
\\=& \frac{1}{\sqrt{2\pi}}w\int_{-\infty}^{w}  (f'(x)-\E f'(Z)
        e^{-x^2/2}dx\bigg\vert_{w=-\infty}^{w=\infty}\\
        &-\frac{1}{\sqrt{2\pi}}\int_{w\in \mathbb{R}}^{ } w  (f'(w)-\E f'(Z))
        e^{-w^2/2}dw
\\=& -\frac{1}{\sqrt{2\pi}}\int_{w\in \mathbb{R}}^{ } w  f'(w)
        e^{-w^2/2}dw.
    \end{aligned}
    \label{eq:1006}
\end{equation}
From \cref{eq:l1-1} and \eq{eq:l1-13}, we have
$$f'(w)=h(w)-\E h(Z)+ w e^{w^2/2} \int_{-\infty}^w (h(x)-\E h(Z)) e^{-x^2/2}dx.$$
Plugging this expression into \cref{eq:1006} and using integration by parts, we have
\begin{equation}
    \begin{aligned}
        \E g(Z)=&-\frac{1}{\sqrt{2\pi}} \int_{-\infty}^{\infty} w^2 \int_{-\infty}^{w}
        (h(x)-\E h(Z)) e^{-\frac{x^{2}}{2}}dxdw
\\&-\frac{1}{\sqrt{2\pi}} \int_{-\infty}^{\infty} w 
        (h(w)-\E h(Z)) e^{-\frac{w^{2}}{2}}dw
        \\=& -\frac{1}{\sqrt{2\pi}}  \frac{w^{3}}{3} \int_{-\infty}^{w}
        (h(x)-\E h(Z)) e^{-\frac{x^{2}}{2}}dx\bigg\vert^{\infty}_{-\infty}
        +\frac{1}{\sqrt{2\pi}} \int_{-\infty}^{\infty} \frac{w^{3}}{3}h(w)
        e^{-\frac{w^{2}}{2}}dw
         \\&-\frac{1}{\sqrt{2\pi}} \int_{-\infty}^{\infty}wh(w)e^{-\frac{x^{2}}{2}}dw
        \\=&\frac{1}{3}\E [(Z^3-3Z)h(Z)].
    \end{aligned}
    \label{eq:10015}
\end{equation}
This proves \cref{eq:l1-12}.

\subsection{Proof of (\ref{eq:t1-1})}\label{sec:3.2}

We will need the following two lemmas.
\begin{lemma}\label{lem:smoothbound}
    Let $h(x)$ be any bounded function and $P(x)$ be any polynomial. Define $h_{1}(x)=\E
    [h(\sigma Z+x) P(Z)] $ and $h_{2}(\sigma)=\E
    [h(\sigma Z) P(Z)]$, where $\sigma>0$ and $Z\sim N(0,1)$. Then for any $k\geq 0$, $h_{1}(x)$ and
    $h_{2}(\sigma)$ are $k$-th order
    differentiable and 
    \begin{equation}
        \begin{aligned}
            \left\lvert\frac{d^{k}}{dx^{k}}h_{1}(x)\right\rvert\leq
            \frac{C_{P,k}\|h\|_\infty}{\sigma^{k}},\quad \forall x\in \IR,
        \end{aligned}
        \label{eq:lem1}
    \end{equation}
\begin{equation}
        \begin{aligned}
            \left\lvert\frac{d^{k}}{d\sigma^{k}}h_{2}(\sigma)\right\rvert\leq
            \frac{\widetilde{C}_{P,k}\|h\|_\infty}{\sigma^{k}},\quad \forall x\in \IR,
        \end{aligned}
        \label{eq:lem5}
    \end{equation}
where $C_{P,k}$ and $\widetilde{C}_{P,k}$ are constants depending only on the polynomial $P$ and $k$. 
\end{lemma}
\begin{proof}[Proof of \cref{lem:smoothbound}]
First, we rewrite $h_{1}(x)$ as the integral form and use change of variable to obtain 
\begin{equation}
    \begin{aligned}
        h_{1}(x)&=
              \frac{1}{\sqrt{2\pi}}\int_{y\in \mathbb{R}} P(y) h(\sigma y+x) e^{-\frac{y^{2}}{2}}dy
              \\&=\frac{1}{\sigma\sqrt{2\pi}}\int_{y\in \mathbb{R}} P\left(\frac{r-x}{\sigma}\right) h(r)
              e^{-\frac{(r-x)^{2}}{2\sigma^{2}}}dr.
    \end{aligned}
    \label{eq:lem2}
\end{equation}
The case $k=0$ follows from the boundedness of moments of $Z$. For $k=1$, taking first order derivative of $h_{1}(x)$, we have  
\begin{equation}
    \begin{aligned}
        \frac{d}{dx}h_{1}(x)&=\frac{1}{\sigma^{2}\sqrt{2\pi}}\int_{y\in \mathbb{R}}
        \left(-P'\left(\frac{r-x}{\sigma}\right)+P\left(\frac{r-x}{\sigma}\right)\frac{r-x}{\sigma}\right) h(r)
              e^{-\frac{(r-x)^{2}}{2\sigma^{2}}}dr
\\&=\frac{1}{\sigma\sqrt{2\pi}}\int_{y\in \mathbb{R}}
        \left(-P'\left(y\right)+yP\left(y\right)\right) h(\sigma
        y+x)
              e^{-\frac{y^{2}}{2}}dy
\\&= \frac{1}{\sigma}\E  \Big[\left(-P'\left(Z\right)+ZP\left(Z\right)\right) h(\sigma
        Z+x)\Big].
    \end{aligned}
    \label{eq:1002}
\end{equation}
From \cref{eq:1002}, we have
\begin{equation}
    \begin{aligned}
        \left\lvert  \frac{d}{dx}h_{1}(x)\right\rvert\leq \frac{\|h\|_\infty}{\sigma} \E \lvert
        -P'\left(Z\right)+ZP\left(Z\right)  \rvert,
    \end{aligned}
    \label{eq:1003}
\end{equation}
and we complete the proof for the case $k=1$. Following similar arguments, we get
\cref{eq:lem1} for $k\geq  2$
and $\cref{eq:lem5} $ for $k\geq 0$.
\end{proof}
\begin{lemma}
    \label{lem:2}
    Let $X_{1},\ldots,X_{n}$ be i.i.d.\ random variables with $\E X_{1}=0$, $\Var(X_1)=1$ and
    $\E \lvert X_1 \rvert^{3}< \infty$. For $h_{1}(x)$ defined in \cref{lem:smoothbound} and any
    $\sigma_{1}\geq 0$, we have
    \begin{equation}
        \begin{aligned}
            \left\lvert\E h_{1}\left( \frac{\sigma_{1}}{\sqrt{n}} \sum^{n}_{i=1} X_{i}\right)-\E h_{1}(
            \sigma_{1} Z
            )\right\rvert\leq \frac{C_P\|h\|_\infty  \sigma_{1}^{3}}{ \sigma^{3}\sqrt{n}}\E( \lvert
            X_{1}\rvert^3+\lvert Z \rvert^{3}),
        \end{aligned}
        \label{eq:1004}
    \end{equation}
    where $Z\sim N(0,1)$ and $C_P$ is a constant depending only on the polynomial $P$ in the definition of $h_1$.
\end{lemma}
\begin{proof}[Proof of \cref{lem:2}]
We prove this lemma by Lindeberg's swapping argument. Let $Z_{1},\ldots,Z_{n}$ be i.i.d.\
$N(0,1)$ random variables. We have
\begin{equation}
    \begin{aligned}
        &\left\lvert\E h_{1}\left( \frac{\sigma_{1}}{\sqrt{n}} \sum^{n}_{i=1} X_{i}\right)-\E h_{1}(
            \sigma_{1} Z
            )\right\rvert
        \\=& \sum^{n}_{k=1} \left\lvert\E h_{1}\left( \frac{\sigma_{1}}{\sqrt{n}}\left( \sum^{k}_{i=1}
            X_{i} + \sum^{n}_{i=k+1} Z_{i}\right) \right)-\E h_{1}\left(
            \frac{\sigma_{1}}{\sqrt{n}}\left( \sum^{k-1}_{i=1}
            X_{i} + \sum^{n}_{i=k} Z_{i}\right) \right)\right\rvert
                    \\\leq &\frac{  \sigma_{1}^{3}}{6 \sqrt{n}} \sup_{x\in \IR}\left\lvert
                    \frac{d^3}{dx^3}h_{1}(x) \right\rvert\E  (\lvert
            X_{1}\rvert^3+\lvert Z \rvert^{3}).
    \end{aligned}
    \label{eq:1005}
\end{equation}
From \cref{eq:lem1}, we complete the proof of this lemma.
\end{proof}

\begin{proof}[Proof of \eq{eq:t1-1}]
Suppose that $X_{1}\sim F(x)$
and $F(x)=p_{1} F_{1}(x)+ (1-p_{1}) F_{2}(x)$, where $F_{1}(x)$ is compactly supported and has density
$p(x)$ bounded away from 0, but $\int_{}^{ }x dF_{1}(x)$ is not equal to $0$ (otherwise \eq{eq:t1-1} follows directly from the condition of \cref{t1}). Suppose that the support of $F_{1}(x)$ is the interval
$[a_{1},a_{2}]$.

\medskip

\textbf{Case 1.} If $0\in (a_{1},a_{2})$, then by truncation, $F_{1}(x)$ has a component $U$ which satisfies
\cref{eq:t1-1}. Furthermore, since $U$ is also a component of $F(x)$, we complete the proof of
\cref{eq:t1-1} in this case.

\medskip

\textbf{Case 2.}
If $0\notin(a_{1},a_{2})$, without loss
of generality, we assume that $a_{1}\geq 0$. Since $\IE X_{1}=0$, there must exists a negative
constant $b$ such that $P(b-\varepsilon\leq X_{1}\leq b)>0$ for any $\varepsilon>0$. Let
$$F_{3}(x)=\frac{F(x)1_{\{x\leq b\}}}{F(b)},$$ we then have 
$$F(x)=p_{1} F_{1}(x)+p_{3}F_{3}(x)+p_{4}F_{4}(x),$$ where $p_{3}=F(b)$ and $p_{4}$ and
$F_{4}(x)$ are determined by this equation. Let $m_{1}$ be the smallest positive integer such
that $m_{1}|a_{1}-a_{2}|>3|b|$. The support of $F_{1}^{*m_{1}}(x)$ is
$[m_{1}a_{1},m_{1}a_{2}]$, where $^*$ denotes convolutions.  Let $m_{2}$ be the smallest positive integer such that
$m_{2}|b|> m_1 a_{1}$. Then we have $$F_{3}^{*m_{2}}(m_{2}b)-F_{3}^{*m_{2}}(m_{2}b-\varepsilon)>0$$ for any
$\varepsilon\geq 0$, and  
\begin{equation}
    (m_{2}-1)|b|\leq m_{1}a_{1} < m_{2}|b| < m_{1}a_{2}. 
    \label{eq:1001}
\end{equation}
By \cref{eq:1001}, there exists a component $\widetilde{F}(x)$ of $F_{1}^{*m_{1}}*F_{3}^{*m_{2}}(x)$ such
that $\widetilde{F}(x)$ is compactly
supported, has density bounded away from 0 and the interior of the support of
$\widetilde{F}(x)$
contains $0$.   Since $F_{1}^{*m_{1}}*F_{3}^{*m_{2}}(x)$ is a component of
$F^{*(m_{1}+m_{2})}(x)$, we have $\widetilde{F}(x)$ is also a component of
$F^{*(m_{1}+m_{2})}(x)$. By the same argument as that in \textbf{Case 1} and let
$m=m_{1}+m_{2}$, we infer that $F^{*m}(x)$ has a component $U$ satisfying \cref{eq:t1-1}.

\medskip

Let $n=km+r$ for some integers $k\geq 1$ (\cref{t1} trivially holds for bounded $n$) and $0\leq r<m$. The case
$r=0$ follows from the proof of \cref{t1}. We now consider the case $r>0$. Let $Z, \tilde{Z}, Z_{1},Z_{2},\cdots,Z_{n}$ be i.i.d.\
$N(0,1)$ random variables and independent of $\{X_1, \dots, X_n\}$. We have
\begin{equation}
    \begin{aligned}
        &\E h(W) -\E h(Z)\\ =&  \E h\left( \frac{1}{\sqrt{n}}\left(\sum^{km}_{i=1}
        X_{i}+\sum^{n}_{i=km+1} X_{i}\right) \right) -\E h\left( \frac{1}{\sqrt{n}}\left(\sum^{km}_{i=1}
                    Z_{i}+\sum^{n}_{i=km+1} X_{i}\right) \right)\\&+\ \E h\left( \frac{1}{\sqrt{n}}\left(\sum^{km}_{i=1}
        Z_{i}+\sum^{n}_{i=km+1} X_{i}\right) \right) -\E h\left( \frac{1}{\sqrt{n}}\left(\sum^{km}_{i=1}
Z_{i}+\sum^{n}_{i=km+1} Z_{i}\right) \right)
                    \\ =:& H_{1}+H_{2}.
    \end{aligned}
    \label{eq:1007}
\end{equation}
For $H_{1}$, conditioning on $\{X_{km+1},\ldots,X_{n}\}$ first and applying
\cref{t1} (i.e. the case $r=0$), we have  
\begin{equation}
    \begin{aligned}
        H_{1}=&   \E h\left(
                    \frac{\sqrt{km}}{\sqrt{n}}\left(\frac{1}{\sqrt{km}}\sum^{km}_{i=1}
                X_{i} \right )+\frac{1}{\sqrt{n}}\sum^{n}_{i=km+1} X_{i} \right) -\E h\left(
                    \frac{\sqrt{km}}{\sqrt{n}}Z+\frac{1}{\sqrt{n}}\sum^{n}_{i=km+1} X_{i}
            \right)
        \\ =&\frac{ \gamma}{6\sqrt{km}} \E (Z^3-3Z) h\left( \frac{\sqrt{km}}{\sqrt{n}}
        Z+\frac{1}{\sqrt{n}}\sum^{n}_{i=km+1} X_{i} \right)+ O(\frac{1}{n})
          \\ =& \frac{ \gamma}{6\sqrt{km}} \E (Z^3-3Z) h\left( \frac{\sqrt{km}}{\sqrt{n}}
          Z+\frac{\sqrt{r}}{\sqrt{n}} \tilde{Z} \right)+ O(\frac{1}{n}),
    \end{aligned}
    \label{eq:1008}
\end{equation}
where we used \cref{lem:2} in the last equality.
Using Gaussian integration by parts (and approximating h by arbitrarily close smooth
functions in the intermediate step), we have
\begin{equation}
    \begin{aligned}
        &\E (Z^3-3Z) h\left( \frac{\sqrt{km}}{\sqrt{n}}
          Z+\frac{\sqrt{r}}{\sqrt{n}} \tilde{Z} \right)
        =-\frac{\sqrt{k^{3}m^{3}}}{\sqrt{n^{3}}}\E  h'''\left( \frac{\sqrt{km}}{\sqrt{n}}
          Z+\frac{\sqrt{r}}{\sqrt{n}} \tilde{Z} \right)
\\=&- \frac{\sqrt{k^{3}m^{3}}}{\sqrt{n^{3}}}\E  h'''\left( Z \right)
= \frac{\sqrt{k^{3}m^{3}}}{\sqrt{n^{3}}}\E [(Z^3-3Z) h\left( Z \right)].
    \end{aligned}
    \label{eq:10012}
\end{equation}
From \cref{eq:1008,eq:10012}, we have
\begin{equation}
    H_{1}= \frac{km\gamma}{6\sqrt{n^{3}}}\E [(Z^3-3Z) h(Z)]+O(\frac{1}{n})=\frac{\gamma}{6\sqrt{n}}\E [(Z^3-3Z) h(Z)]+O(\frac{1}{n}).
    \label{eq:10013}
\end{equation}
For $H_{2}$, let 
$$h_{1}(x)= \E  h\big(\frac{\sqrt{km}}{\sqrt{n}} Z+ x\big).$$
Then, by Lindeberg's swapping argument, 
\begin{equation}
    \begin{aligned}
        H_{2}&=\E h_{1}\left( \sum^{n}_{i=km+1} \frac{X_{i}}{\sqrt{n}}\right) -\E h_{1}\left( \sum^{n}_{i=km+1} \frac{X_{i}}{\sqrt{n}}\right)
\\&= \sum^{n}_{j=km+1} \Biggl\{\E h_{1}\left( \sum^{j}_{i=km+1} \frac{X_{i}}{\sqrt{n}}+\sum^{n}_{i=j+1}
\frac{Z_{i}}{\sqrt{n}}\right) -\E h_{1}\left( \sum^{j-1}_{i=km+1} \frac{X_{i}}{\sqrt{n}}+\sum^{n}_{i=j} \frac{Z_{i}}{\sqrt{n}}\right)\Biggr\}
\\&= \sum^{n}_{j=km+1} \Biggl\{\E h_{1}'\left( \sum^{j-1}_{i=km+1} \frac{X_{i}}{\sqrt{n}}+\sum^{n}_{i=j+1}
\frac{Z_{i}}{\sqrt{n}}\right)\left(\frac{X_{j}}{\sqrt{n}}-\frac{Z_{j}}{\sqrt{n}}\right) 
\\&\quad\quad\quad\quad\quad+\E h_{1}''\left( \sum^{j-1}_{i=km+1} \frac{X_{i}}{\sqrt{n}}+\sum^{n}_{i=j+1}
\frac{Z_{i}}{\sqrt{n}}\right)\left(\frac{X_{j}^{2}}{2{n}}-\frac{Z_{j}^{2}}{2{n}}\right)
\\&\quad\quad\quad\quad\quad+ O(1) \lVert h_{1}''' \rVert_\infty \E\left(
\left\vert\frac{X_{i}^{3}}{{n^{3/2}}}\right\vert+\left\vert\frac{Z_{i}^{3}}{{n^{3/2}}}\right\vert\right)\Biggr\}
\\&= O(\frac{r}{n^{3/2}})(\E|X_{1}|^{3}+\E|Z|^{3}),
    \end{aligned}
    \label{eq:10011}
\end{equation}
 where we use \cref{lem:smoothbound} in the last equality.
 Combining \cref{eq:10013,eq:10011}, we have
\begin{equation}
    \begin{aligned}
        H_{1}+H_{2}=\frac{\gamma}{6\sqrt{n}}\E (Z^3-3Z)h(Z)+O(\frac{1}{n}).
    \end{aligned}
    \label{eq:10010}
\end{equation}
Thus, we have proved for the case $r>0$.
\end{proof}

\subsection{Proof of (\ref{eq:t1-7})}\label{sec:3.3}

\eq{eq:t1-7} follows immediately from \cref{lem:smoothbound,lem:2}.

\subsection{Proof of Lemma \ref{lem:binomial}}\label{sec:3.4}
In this subsection, we use $O(1)$ to denote a quantity which is bounded in absolute value by a universal constant. We will use the following lemma. 
\begin{lemma}
    Let $n\geq 1$ and  $X_{1},\ldots,X_{n}$ be i.i.d.\ $Ber(1/2)$ random variables. Then we have, for any
    constant $x$ in the set $A_n:=\{\frac{2z}{\sqrt{n}}-\sqrt{n},z=0,\cdots,n\}\cap
    [-n^{1/4}/8,{n^{1/4}}/8] $, 
    \begin{equation}
        \begin{aligned}
            \P\left(\frac{2}{\sqrt{n}}\sum^{n}_{i=1} (X_{i}-\frac{1}{2})=x\right)=&
           \frac{2}{\sqrt{2\pi n}} 
           e^{-\frac{x^2}{2}} \left(1+O\left(\frac{1+x^4}{n}\right)\right)\\
            =& \frac{1}{\sqrt{2\pi}} \int_{x-\frac{1}{\sqrt{n}}}^{x+\frac{1}{\sqrt{n}}}
            e^{-\frac{y^2}{2}} dy\left(1+O\left(\frac{1+x^4}{n}\right)\right).
        \end{aligned}
        \label{eq:10014}
    \end{equation}
    \label{lem:4}
\end{lemma}

\begin{proof}[Proof of \cref{lem:4}]
In this proof, we will use the following inequality:
    \begin{equation}
        \begin{aligned}
            \alpha-\frac{1}{2}\alpha^{2} +\frac{1}{3} \alpha^{3}- \alpha^{4}\leq \log(1+\alpha) \leq
            \alpha-\frac{1}{2}\alpha^{2}+\frac{1}{3}\alpha^{3},\quad \forall \lvert \alpha \rvert\leq \frac{1}{8},
        \end{aligned}
        \label{eq:10020}
    \end{equation}
which is proved by Taylor's expansion.
Without loss of generality, we assume that $x\in A_n$, $x\geq 0$, and let $z$ be the integer such that $\frac{2z}{\sqrt{n}}-\sqrt{n}=x$.
Assume $n$ is sufficiently large (otherwise \eq{eq:10014} is trivial) so that $z$ and $n-z$ are sufficiently large.
Then 
\begin{equation}
    \begin{aligned}
        \P\left(\frac{2}{\sqrt{n}}\sum^{n}_{i=1} (X_{i}-\frac{1}{2})=x\right)=\P\left( \sum^{n}_{i=1}
        X_{i}=z\right)=\frac{n!}{z!(n-z)!}\frac{1}{2^{n}}.
    \end{aligned}
    \label{eq:10016}
\end{equation}
Using the Stirling formula
\begin{equation*}
    \begin{aligned}
        n!=\sqrt{2\pi n} \left( \frac{n}{e} \right)^{n} \left(1+ O\left(\frac{1}{n}
        \right)\right),
    \end{aligned}
\end{equation*}
we have 
\begin{equation}
    \begin{aligned}
        \frac{n!}{z!(n-z)!}\frac{1}{2^{n}}=\frac{1}{2^n}\frac{\sqrt{2\pi n} n^{n} }{\sqrt{2\pi z}
            \sqrt{ 2\pi (n-z)}z^{z} (n-z)^{(n-z)}}  \frac{\left( 1+O\left(
                    \frac{1}{n}
                \right) \right)}{\left( 1+O\left( \frac{1}{z} \right) \right)\left(
                1+O\left( \frac{1}{n-z} \right)
    \right)}.
    \end{aligned}
    \label{eq:10018}
\end{equation}
Plugging $z=\frac{1}{2}\sqrt{n} x+\frac{n}{2}$ into \cref{eq:10018} and using $x\in [-n^{1/4}/8, n^{1/4}/8]$, 
we have 
\begin{equation}
    \begin{aligned}
       & \P\left(\frac{2}{\sqrt{n}}\sum^{n}_{i=1}
    (X_{i}-\frac{1}{2})=x\right)\\=&\frac{2}{\sqrt{2 \pi n} }
            \left(1-\frac{x^{2}}{n}\right)^{-\frac{1}{2}} \left( 1+\frac{x}{\sqrt{n}}
            \right)^{-\frac{\sqrt{n}}{x} \frac{x^2}{2}} \left( 1-\frac{x}{\sqrt{n}}
            \right)^{\frac{\sqrt{n}}{x} \frac{x^2}{2}} \left(1-\frac{x^{2}}{n}
    \right)^{-\frac{n}{x^{2}} \frac{x^{2}}{2}}\\&\quad\quad\quad\times\left( 1+O\left(
        \frac{1}{n}+\frac{1}{n+\sqrt{n}x}+ \frac{1}{n-\sqrt{n}x} \right) \right)
\\=&\frac{2}{\sqrt{2 \pi n} }
           \left( 1+\frac{x}{\sqrt{n}}
            \right)^{-\frac{\sqrt{n}}{x} \frac{x^2}{2}} \left( 1-\frac{x}{\sqrt{n}}
            \right)^{\frac{\sqrt{n}}{x} \frac{x^2}{2}} \left(1-\frac{x^{2}}{n}
    \right)^{-\frac{n}{x^{2}} \frac{x^{2}}{2}}\\&\quad\quad\quad\times\left( 1+O\left(
\frac{1}{n}+\frac{1}{n+\sqrt{n}x}+ \frac{1}{n-\sqrt{n}x}+\frac{x^{2}}{n} \right) \right).
    \end{aligned}
    \label{eq:10019}
\end{equation}
Because $\lvert x/\sqrt{n}\vert\leq \lvert x/n^{1/4}\vert\leq 1/8 $, applying \cref{eq:10020}, we have 
\begin{align}\label{eq:10022}
   \nonumber \left( 1+\frac{x}{\sqrt{n}} \right)^{\frac{\sqrt{n}}{x}} =  &
e^{1-\frac{x}{2\sqrt{n}}+O\left( \frac{x^{2}}{n}\right)},\\
   \nonumber \left( 1-\frac{x}{\sqrt{n}} \right)^{\frac{\sqrt{n}}{x}} =  &
e^{-1-\frac{x}{2\sqrt{n}}+O\left( \frac{x^{2}}{n}\right)},\\
   \left( 1-\frac{x^{2}}{n} \right)^{\frac{n}{x^{2}}}=&e^{-1+O\left( \frac{x^{2}}{n} \right)}.
\end{align}
Plugging \cref{eq:10022} into \cref{eq:10019}, we obtain
\begin{equation}
    \begin{aligned}
        \P\left(\frac{2}{\sqrt{n}}\sum^{n}_{i=1}
            (X_{i}-\frac{1}{2})=x\right)=\frac{2}{\sqrt{2\pi
        n}}e^{-\frac{x^{2}}{2}}\left(1+O\left(\frac{1+x^{4}}{n}\right)\right).
    \end{aligned}
    \label{eq:10023}
\end{equation}
For the integral on the right hand side of \cref{eq:10014}, we have
\begin{equation}
    \begin{aligned}
        &\quad\int_{x-\frac{1}{\sqrt{n}}}^{x+\frac{1}{\sqrt{n}}} e^{-\frac{y^{2}}{2}} dy
        =
        e^{-\frac{x^{2}}{2}}\int_{x-\frac{1}{\sqrt{n}}}^{x+\frac{1}{\sqrt{n}}} e^{-\frac{y^{2}-x^{2}}{2}} dy
        \\
        &= e^{-\frac{x^{2}}{2}}\int_{-\frac{1}{\sqrt{n}}}^{\frac{1}{\sqrt{n}}} e^{-xr-\frac{r^{2}}{2}} dr
        = 
        e^{-\frac{x^{2}}{2}}  \int_{-\frac{1}{\sqrt{n}}}^{\frac{1}{\sqrt{n}}} \Big[1-xr-r^{2}
        + O\left(xr+\frac{r^{2}}{2}\right)^{2}\Big] dr
       \\
        &= \frac{2}{\sqrt{n}} e^{-\frac{x^{2}}{2}}\left(1+O\left( \frac{1+x^{2}}{n} \right)\right).
    \end{aligned}
    \label{eq:10024}
\end{equation}
Combining \cref{eq:10024,eq:10023}, we complete the proof. 
\end{proof}
We now prove \cref{lem:binomial}.
\begin{proof} [Proof of \cref{lem:binomial}] 
Denote $$A_{l}={ \{{2z}/{\sqrt{l}}-\sqrt{l},z=0,\cdots,l\}\cap
[-l^{1/4}/8,{l^{1/4}}/8] }$$ and $$\widetilde{A}_{l}={ \{{2z}/{\sqrt{l}}-\sqrt{l},z=0,\cdots,l\}\cap
[-l^{1/4}/8,{l^{1/4}}/8]^{c} }.$$ 
From the condition that $h$ equals a constant in $(z-1/2, z+1/2)$ for $z\in \IZ$ and Gaussian tail bounds,
we have
\begin{equation}
    \begin{aligned}
        &\E h(S)-\E h\left(\frac{l}{2}+Z\sqrt{\frac{l}{4}}  \right)\\
        =&\E h\left( \frac{S-\frac{l}{2}}{\frac{\sqrt{l}}{2}} \frac{\sqrt{l}}{2}+ \frac{l}{2}\right)-\E h\left(\frac{l}{2}+Z\sqrt{\frac{l}{4}} \right)
        \\=& \sum^{}_{x\in A_{l}}  h\left(\frac{\sqrt{l}}{2} x+ \frac{l}{2}\right) \left ( P\left(\frac{S-l/2}{\sqrt{l}/2}=x\right)-
        \frac{1}{\sqrt{2 \pi}} \int_{x-1\sqrt{l}}^{x+1\sqrt{l}} e^{-\frac{y^2}{2}}dy \right )
         \\&   +\sum^{}_{x\in \widetilde{A}_{l}}  h\left(\frac{\sqrt{l}}{2} x+ \frac{l}{2}\right)  \left (P\left(\frac{S-l/2}{\sqrt{l}/2}=x\right)-
         \frac{1}{\sqrt{2 \pi}} \int_{x-1\sqrt{l}}^{x+1\sqrt{l}} e^{-\frac{y^2}{2}}dy \right )+O(\frac{1}{l})
            \\=:&R_1+R_2+O(\frac{1}{l}).
    \end{aligned}
    \label{eq:10026}
\end{equation}
From \cref{lem:4}, we have 
\begin{equation}
    \begin{aligned}
       & \sum^{ }_{x\in A_l } \left\lvert P\left(\frac{S-l/2}{\sqrt{l}/2}=x\right)-
        \frac{1}{\sqrt{2 \pi}} \int_{x-1\sqrt{l}}^{x+1\sqrt{l}} e^{-\frac{y^2}{2}}dy \right\rvert
        \\=& O(\frac{1}{l})\sum^{ }_{x\in A_l }  
        (1+x^4) \frac{1}{\sqrt{2 \pi}} \int_{x-1\sqrt{l}}^{x+1\sqrt{l}} e^{-\frac{y^2}{2}}dy
= O(\frac{1}{l}).
\end{aligned}
    \label{eq:10025}
\end{equation}
From \cref{eq:10025}, we have
\ben{\label{eq:10030}R_1=O(\frac{1}{l}).}
For $R_2$,  we have
    \begin{align*}
        |R_2|\leq& \sum^{}_{x\in \widetilde{A}_{l}}    \left (P\left(\frac{S-l/2}{\sqrt{l}/2}=x\right)+
         \frac{1}{\sqrt{2 \pi}} \int_{x-1\sqrt{l}}^{x+1\sqrt{l}} e^{-\frac{y^2}{2}}dy
         \right ).
    \end{align*}
For $x\in \widetilde A_l$, we have $|x|> l^{1/4}/8$. Therefore, from binomial and Gaussian tail bounds, we have
    \begin{equation}
        R_2=O(\frac{1}{l}).
        \label{eq:10029}
    \end{equation}
    Combining \cref{eq:10026,eq:10030,eq:10029}, we complete the proof.
\end{proof}

\subsection{Stein kernel for Gaussian mixtures}\label{sec:3.5}

The following lemma was used in the proofs of \cref{t1,t2} to apply the Stein kernel bound \eq{eq:l1-0} to Gaussian mixtures.
\begin{lemma}\label{lem:mixture}
Let 
\be{
Y=
\begin{cases}
Z_1\sim N(\mu_1, \sigma_1^2),& \text{with probability}\  p,\\
Z_2 \sim N(\mu_2, \sigma_2^2), & \text{with probability}\  1-p.
\end{cases}
}
Let $\tau$ be its Stein kernel.
Then we have $\|\tau\|_\infty\leq C$, where $C$ is a positive constant depending only on $p, \mu_1, \sigma_1^2, \mu_2, \sigma_2^2$.
\end{lemma}

\begin{proof}[Proof of \cref{lem:mixture}]
Assume without loss of generality that $0<p<1$. Otherwise, the lemma follows from the fact that the Stein kernel for a Gaussian variable equals its variance.
We first consider the case $x\geq 0$.
Let $\phi_1(x)$ and $\phi_2(x)$ be the density of $N(\mu_1, \sigma_1^2)$ and $N(\mu_2, \sigma_2^2)$, respectively. Then, $Y$ has density
\be{
p(x)=p\phi_1(x)+q\phi_2(x),\ q=1-p.
}
From the expression of Stein kernel in \eq{eq:saumard}, we have
\be{
\tau(x)=\frac{1}{p\phi_1(x)+q\phi_2(x)}\int_x^\infty (y-p \mu_1-q\mu_2) (p\phi_1(y)+q\phi_2(y))dy.
}
Using 
\be{
\int_x^\infty (y-\mu_i) \phi_i(y)dy=\sigma_i^2 \phi_i(x),\ i=1,2,
}
we obtain
\be{
\tau(x)=\frac{p\sigma_1^2 \phi_1(x)+q\sigma_2^2 \phi_2(x) +q(\mu_1-\mu_2) \int_x^\infty \phi_1(y)dy +p(\mu_2-\mu_1) \int_x^\infty \phi_2(y)dy    }{p\phi_1(x)+q\phi_2(x)},
}
which is bounded because $\int_x^\infty \phi_i(y)dy$ decays proportional to $\phi_i(x)/x$ for large $x$ and $i=1,2$.

The case $x<0$ is proved similarly using the alternative expression of the Stein kernel
\be{
\tau(x)=-\frac{1}{p(x)}\int_{-\infty}^x (y-\mu) p(y)dy.
}
\end{proof}

\subsection{Existence of Bernoulli component}\label{sec:GCD}

Recall the condition of \cref{t2}: The support of $X_1$ is $\{s_0,s_1,s_2,\dots\}$ such that $\{|s_i-s_0|,i\geq 1\}$ has the greatest common divisor 1. By successively looking for numbers that decrease the common divisor, this condition implies that there exists a finite $r$ such that $\{|s_i-s_0|, 1\leq i\leq r\}$ has the greatest common divisor 1. By shifting invariance, we assume without loss of generality that $0=s_0<s_1<\dots<s_r$. Repeatedly using B\'ezout's identity, there exist non-zero integers $m_1,\dots, m_r$ such that
\be{
m_1 s_1+\dots +m_r s_r=1.
}
Let $X_{ij}, i,j\geq 1$ follow the same distribution as $X_1$. For $1\leq i\leq r$, let 
\be{
t_i=
\begin{cases}
0, & \text{if}\  m_i>0\\
-s_i, & \text{if}\  m_i<0,
\end{cases}
}
and
\be{
Y_i=\sum_{j=1}^{|m_i|} (X_{ij}+t_j).
}
Then, $\sum_{i=1}^r Y_i$ can take values $0$ and $1$. This implies that $\sum_{i=1}^m X_i-z$ has a $Ber(1/2)$ component for $m=|m_1|+\dots+|m_r|$ subject to shifting by a fixed integer.

\section*{Acknowledgements}

Fang X. was partially supported by Hong Kong RGC GRF 14302418, 14305821 and a CUHK direct grant.

\bibliographystyle{apalike}
\bibliography{reference}

\end{document}